\documentclass[a4paper,12pt]{amsart}
\usepackage{amssymb}
\usepackage{ifthen}
\usepackage{cite}
\usepackage[dvips]{graphicx}
\nonstopmode \numberwithin{equation}{section}

\usepackage{amssymb}
\usepackage{ifthen}
\usepackage{graphicx}
\usepackage{amsmath}
\usepackage[T1]{fontenc} %skandit
\usepackage[utf8]{inputenc}
\usepackage[usenames,dvipsnames]{color}
\usepackage{color}
\usepackage[english]{babel}
\usepackage{fancyhdr}
\usepackage{fancybox}
\usepackage{tikz}
\nonstopmode \numberwithin{equation}{section}
\theoremstyle{plain}

\newtheorem{thm}{Theorem}[section]
\newtheorem{cor}{Corollary}[section]
\newtheorem{lem}{Lemma}[section]
\newtheorem{note}{Note}[section]
\newtheorem{ques}{Question}[section]
\newtheorem{prop}{Proposition}

\newtheorem{conj}{Conjecture}

\theoremstyle{definition}
\newtheorem{defn}{Definition}[section]
\newtheorem{exm}{Example}[section]
\newtheorem{prob}{Problem}
\newtheorem{rem}{Remark}[section]

\newcounter{minutes}
\divide\time by 60
\newcounter{hours}
\multiply\time by 60
\addtocounter{minutes}{-\time}
\newcounter {own}
\def\theown {\thesection       .\arabic{own}}

\newenvironment{pf}[1][]{%
 \vskip 3mm
 \noindent
 \ifthenelse{\equal{#1}{}}%
  {{\slshape Proof. }}%
  {{\slshape #1.} }%
 }%
{\qed\bigskip}

\newcounter{alphabet}

\newcommand{\T}{{\mathcal T}}

\def\be{\begin{equation}}
\def\ee{\end{equation}}

\newcommand{\bee}{\begin{enumerate}}
\newcommand{\eee}{\end{enumerate}}

\newcommand{\blem}{\begin{lem}}
\newcommand{\elem}{\end{lem}}
\newcommand{\bthm}{\begin{thm}}
\newcommand{\ethm}{\end{thm}}
\newcommand{\bcor}{\begin{cor}}
\newcommand{\ecor}{\end{cor}}
\newcommand{\beg}{\begin{examp}}
\newcommand{\eeg}{\end{examp}}
\newcommand{\begs}{\begin{examples}}
\newcommand{\eegs}{\end{examples}}

\newcommand{\bdefn}{\begin{defn}}
\newcommand{\edefn}{\end{defn}}

\newcommand{\bprob}{\begin{prob}}
\newcommand{\eprob}{\end{prob}}
\newcommand{\bei}{\begin{itemize}}
\newcommand{\eei}{\end{itemize}}

\newcommand{\bcon}{\begin{conj}}
\newcommand{\econ}{\end{conj}}
\newcommand{\bcons}{\begin{conjs}}
\newcommand{\econs}{\end{conjs}}
\newcommand{\bprop}{\begin{prop}}
\newcommand{\eprop}{\end{prop}}
\newcommand{\br}{\begin{rem}}
\newcommand{\er}{\end{rem}}
\newcommand{\brs}{\begin{rems}}
\newcommand{\ers}{\end{rems}}
\newcommand{\bo}{\begin{obser}}
\newcommand{\eo}{\end{obser}}
\newcommand{\bos}{\begin{obsers}}
\newcommand{\eos}{\end{obsers}}
\newcommand{\bpf}{\begin{pf}}
\newcommand{\epf}{\end{pf}}
\newcommand{\ba}{\begin{array}}
\newcommand{\ea}{\end{array}}
\newcommand{\beq}{\begin{eqnarray}}
\newcommand{\beqq}{\begin{eqnarray*}}
\newcommand{\eeq}{\end{eqnarray}}
\newcommand{\eeqq}{\end{eqnarray*}}

\begin{document}

\title{Meromorphic Solutions of Difference Equations Involving Borel and Nevanlinna Exceptional Values}

\author{Molla Basir Ahamed$^*$}
\address{Molla Basir Ahamed,
	Department of Mathematics, Jadavpur University, Kolkata-700032, West Bengal, India.}
\email{mbahamed.math@jadavpuruniversity.in}

\author{Vasudevarao Allu}
\address{Vasudevarao Allu,
	Department of Mathematics, School of Basic Sciences,
Indian Institute of Technology Bhubaneswar,
Bhubaneswar-752050, Odisha, India.}
\email{avrao@iitbbs.ac.in}

\subjclass[{AMS} Subject Classification:]{30D35, 30D20}
\keywords{Meromorphic function, shared values, difference operators, solutions of difference equation, Borel exceptional value, Nevanlinna exceptional value, finite order}

\def\thefootnote{}
\footnotetext{ {\tiny File:~\jobname.tex,
printed: \number\year-\number\month-\number\day,
          \thehours.\ifnum\theminutes<10{0}\fi\theminutes }
} \makeatletter\def\thefootnote{\@arabic\c@footnote}\makeatother

\begin{abstract}
The existence of meromorphic solutions to various difference equations has been extensively studied in recent years, the precise functional forms of such solutions—particularly when the function and its difference operators share values—remain largely unexplored. This paper addresses this research gap by investigating the sharing value problem between finite-order meromorphic functions $f(z)$ and their linear difference operators $L_c^n(f)$. Specifically, we consider functions having Borel or Nevanlinna exceptional values. We prove not only the existence but also characterize the explicit general meromorphic solutions to the difference equation $L_c^n(f) \equiv Af$ for $A \in \mathbb{C} \setminus \{0\}$. To validate our main results and demonstrate the necessity of our conditions, we provide several concrete examples. Furthermore, we investigate the existence and nature of both rational and transcendental meromorphic solutions for the second-order difference equation $b_2(z)f(z+2\eta) + b_1(z)f(z+\eta) + b_0(z)f(z) = b(z)$ with polynomial coefficients.
\end{abstract}

\maketitle
\pagestyle{myheadings}
\markboth{Molla Basir Ahamed and Vasudevarao Allu}{Sharing value problem and meromorphic solutions of difference equations}

\section{\bf Introduction}
\par The complex oscillation theory of meromorphic solutions of differential equation is an important topic in complex analysis (see \cite{Yan & Yi-2006}), where Nevanlinna theory is an effective research tool. In the recent years, research on finding meromorphic solutions of difference equations in view of shared values are extensively studied (see e.g. \cite{Aha & RM & 2019,Hal & Kor-AASFM-2006,Hal & Kor - JPA-2007,Risto-TAMS-2020}) after the development of the difference analogue lemma of logarithmic derivatives of meromorphic functions. However, little is known about the precise form of the entire or meromorphic solutions of the difference equations that eventually appear as conclusions of the results when shared value considered. With the recent development regarding sharing value problems between meromorphic functions and finding the exact class of all meromorphic solutions, we refer the reader to the article  \cite{Ahamed-IJPAM-2022}. Focusing on shared value problems for meromorphic functions, using Nevanlinna theory as theoretical tool, in this paper, we study certain properties of meromorphic solutions of difference equations. Moreover, our prime concern will be to find the precise form of the meromorphic solutions to the results we prove in this paper. \vspace{1.2mm} 

\par Before proceeding, we spare the reader for a moment and assume some familiarity with the basics of Nevanlinna theory of meromorphic functions in $ \mathbb{C} $ such as the first and second main theorems, and the usual notations such as the \emph{characteristic function} $ T(r,f) $, \emph{proximity function} $ m(r,f) $ and \emph{counting function} $ N(r,f)\; (\overline{N}(r,f)) $ of poles, for more details, we refer the reader to glance through \cite{Hay-1964,Yan - 1993}. In addition, we use the notation $ \lambda(f) $ for the exponent of convergence of the sequence of zeros of a meromorphic function $ f $, and $ \sigma(f) $ to denote the order of growth of $ f $. Finally, $ \sigma_2(f) $ denotes hyper-order (see \cite{Yan & Yi-2003}) of $ f $ which is defined by 
\begin{equation*}
	\sigma_2(f)=\limsup_{r\rightarrow\infty}\frac{\log\log T(r,f)}{\log r}.
\end{equation*}
\par Two non-constant meromorphic functions $ f $ and $ g $ are said to share the value $ a $ $ CM $ $ (IM) $ if $ f-a $ and $ g-a $ have the same set of zeros counting multiplicities (ignoring multiplicities). In addition, we say that $ f $ and $ g $ share the value $ \infty $ $ CM (IM) $ if $ f $ and $ g $ have the same set of poles counting multiplicities (ignoring multiplicities). Shared value problems of meromorphic functions and different aspects of it gained a valuable space in the literature of meromorphic functions (see for example, \cite{Gun - TAMS-1983,Gun - TAMS-1987}) and references therein. \vspace{1.5mm}

\par It was proved in \cite{Whit - 1935} that  for a given entire function $ \Psi(z) $  with order $ \sigma(\Psi)=\sigma $, the equation $ f(z+1)=\Psi(z)f(z) $
admits a meromorphic solution of order $ \sigma(f)\leq \sigma+1. $ Later, {Bank and Kaufman} \cite{Ban & Kau - 1976} proved that the difference equation $ 	f(z+1)-f(z)=R(z) $ has non-trivial meromorphic solution $ y(z) $ such that $ T(r,f)=O(r). $ However, {Yanagihara} \cite{Func Ekvac - 1980} studied the difference equation $f(z+1)=R(f(z)) $ and proved that it has a non-trivial meromorphic solution in $ \mathbb{C} $, where $ R $ is a rational function. Moreover, {Shimomura} \cite{Shi - 1981} proved that the difference equation $ f(z+1)=P(f(z)) $, where $ P $ is a polynomial, admits a non-trivial entire solution. The difference analogue to Nevanlinna theory have been established  by {Halburd and Korhonen} \cite{Hal & Kor-AASFM-2006,Hal & Kor - JPA-2007}, {Chiang and Feng} \cite{Chi & Fen-RM-2008}, independently, and improved by {Halburd \emph{et al.}} \cite{H & K & T - TAMS -2014} from finite order meromorphic functions to infinite order (hyper-order strictly less than $ 1 $), which is a powerful theoretical tool to study complex difference equations (see \cite{Che & Yi-RM-2013}), the uniqueness problem of meromorphic functions taking into shifts (see \cite{Ban & Aha & MS & 2019,Chi & Fen-RM-2008}) or difference operators (see \cite{Aha & RM & 2019,Ahamed-IJPAM-2022,Che & Yi-RM-2013}). The uniqueness problem  value sharing by meromorphic function $ f(z) $ with its shift $ f(z+c) $ or difference operator $ \Delta_{c}f=f(z+c)-f(z) $ and the corresponding solution of the complex differences equations have been extensively studied. For the combination of shifts operators $ f(z+nc), \ldots $, $ f(z+c) $ with $ f(z) $, however, there have been a few investigations on finding the general (entire or meromorphic) solutions of complex difference equations.\vspace{1.5mm}

\par The study of non-constant meromorphic solutions of complex difference equations has a long history. For the last few decades, there has been renewed interest in studying the nature of solutions of difference equations in the complex plane $ \mathbb{C} $  (see \cite{Che - 2011,Chen JMMA 2013,Chi & Fen-RM-2008,Hei & Kor & Lai -JMMA-2009,Risto-TAMS-2020,Jia & Chen 2013,Li-Hao-Yi-MJM-2021,Wang JMMA 2011}). In particular and most noticeably, the results of {Ablowitz} \emph{at al.} \cite{Abl & Hal & Her 2000}, using the idea of the order of growth of meromorphic functions in the sense of classical Nevanlinna theory \cite{Hay-1964} as a detection of solvability of the second order non-linear difference equation in $ \mathbb{C} $. In contrast to the differential equation, it is well-known (see \cite{Abl & Hal & Her 2000}) that non-linear difference equations often admit global meromorphic solutions and hence one can apply difference analogue theories of Nevanlinna to characterize nature of the solutions. Therefore, it is difficult to study the solvability of certain difference equations in a complex plain $ \mathbb{C} $. The logarithmic derivative lemma and its difference analogue lemma, play a crucial role in this study. Considering the problem of sharing one value $ CM $ between  finite order entire function $ f $ and its difference operators $ \Delta_{c}f $,  {Chen and Yi} \cite{Che & Yi-RM-2013} have established a relationship between them. 
\begin{thm}\cite{Che & Yi-RM-2013}\label{th-1.1}
Let $ f $ be a finite order transcendental entire function which has a finite Borel exceptional value $ a $, and let $ \eta\in\mathbb{C} $ be a constant such that $ f(z+c)\not\equiv f(z) $. If $ \Delta_{c}f $ and $ f $ share the value $ a $ $ CM $, then $ 	a=0\;\;\text{and}\;\; \Delta_{c}f\equiv {A} f, $ where $ {A} $ is a non-zero constant.
\end{thm}  
\begin{rem}
In Theorem \ref{th-1.1}, we see that if $ f $ has a non-zero finite Borel exceptional value $ a $, then $ a $ is not shared by $ \Delta_{c}f $ and $ f $.  For example, if we consider $ f(z)=be^z+a $, then it follows that $ f $ has the Borel exceptional value $ a $. Clearly, for $ c\neq 2k\pi i $, $ k\in\mathbb{Z} $, the value $ a $ is not shared by $ \Delta_{c}f=b(e^{c}-1)e^z $ and $ f(z) $.
\end{rem}
{In \cite[Remark 1.2]{Che & Yi-RM-2013}, the authors illustrate that for $f(z)=e^z$, the difference operator satisfies $\Delta_c f = Af$ with $A=e^c-1$, suggesting that $A$ depends solely on $c$. However, we demonstrate via Examples \ref{Exm-1.1} and \ref{Exm-1.3} that such conclusions often depend on specific functional structures and are not universally applicable. While existing literature (e.g., \cite{Che & Yi-RM-2013}) primarily addresses the existence or uniqueness of solutions under sharing conditions, it frequently fails to determine the precise functional form of $f(z)$. This discrepancy highlights the necessity of our approach, which provides a complete characterization of the solution space for the general operator $L_c^n(f)$ by identifying the underlying periodic coefficients.}
\begin{exm}\label{Exm-1.1}
Let $ f(z)=\displaystyle e^{\frac{z \ln 5}{\omega}} \exp\left(({2\pi i}/{\omega}) z\right) $, choose, $ {A}=4 $ and $ c=\omega $, where $ \omega $ is a non-real cube root of unity. Clearly,  $ \Delta_{c}f\equiv 4 f $ but the constant $ {A}=4 $ is independent of $ c=\omega $. Moreover, we consider the functions
\begin{align*}
	f_1(z)=\displaystyle\left(\pi^{\pi}\right)^{z/{i}} \sin\left(({2\pi}/{i})z\right)\; \mbox{and}\; f_2(z)=\displaystyle\left(\pi^{\pi}\right)^{z/{i}} \cos\left(({2\pi}/{i})z\right),
\end{align*} $ c=i $ and $ \mathcal{A}=\pi^{\pi}-1 $. Evidently,  $ \Delta_{c}f_j\equiv \left(\pi^{\pi}-1\right) f_j $ for $ j=1, 2 $, but the constant $ {A}=\pi^{\pi}-1 $ is independent of $ c $.
\end{exm}
Indeed, a result established in \cite{Che & Yi-RM-2013} for meromorphic functions shows that the restriction `\textit{$f$ has a finite Borel exceptional value}' in Theorem \ref{th-1.1} can be omitted.
\begin{thm}\cite{Che & Yi-RM-2013}\label{th-1.2}
\textrm{Let $ f $ be a transcendental meromorphic function such that its order of growth $ \sigma(f) $ is not an integer or infinite, and let $ \eta\in\mathbb{C} $ be a constant such that $ \Delta_{c}f\not\equiv 0 $. If $ \Delta_{c}f $ and $ f $ share three distinct values $ a, b$ and $ \infty $ $ CM $, then $ f(z+c)=2f(z) $.}
\end{thm}

\begin{exm}\cite{Che & Yi-RM-2013}
Suppose that $ f(z)=e^z D(z) $, where $ D(z) $ is a periodic function with period $ c=\log 2 $, and its order of growth $ \sigma(D) $ is not an integer or infinite. It follows that $ \Delta_{c} $ and $ f(z) $ share the values $ 1 $ and $ 2 $ $ CM $ and $ f(z+c)=2f(z) $ holds good.
\end{exm}
The above example confirms the existence of a function satisfying Theorem \ref{th-1.2}; however, the general solution of $f(z+c)=2f(z)$ remained unexplored. Subsequently, Zhang and Liao \cite{Zha & Lia-2014} established the precise form of these solutions, thereby improving upon the result of Chen and Ye \cite{Che & Yi-RM-2013}.
\begin{thm}\cite{Zha & Lia-2014}\label{th-1.3}
	Let $ f(z) $ be a transcendental entire function of finite order, and  $ a $, $ b $ be two distinct constants. If $ \Delta_1f(\not\equiv 0) $ and $ f $ share $ a,\;b $ $ {CM} $, then $ \Delta_1f\equiv f $. Furthermore, $ f(z) $ assumes the form $ f(z)=2^{z}h(z) $, where $ h(z) $ is a periodic entire function with period $ 1 $.   
\end{thm}
We have the following remark on {Theorem \ref{th-1.3}}.
\begin{rem}
The condition of sharing `two values' in Theorem \ref{th-1.3} cannot be relaxed to sharing `one value'. For example, consider $f(z)=Ae^{\pi iz/c}$. It is easily verified that $f$ and $\Delta_c f$ share $ 0 $ CM, yet$ f\not\equiv \Delta_c f $. Furthermore, Theorem \ref{th-1.3} is satisfied not only by functions of finite order but also by certain functions of infinite order.\vspace{1.2mm}

\noindent{\bf A.} Let $ f(z)=2^{z}e^{c\cos(2\pi z)+d} $, where $c \neq 0$ and $d \in \mathbb{C}$. It follows that $\Delta_1 f$ and $f$ share any two distinct values $a, b$ CM, and we have $\Delta_1 f \equiv f$. Note that $f$ is of the form $ f(z)=2^{z}h(z) $, where $h(z)=e^{c\cos(2\pi z)+d}$ is a periodic entire function with period $1$.\vspace{1.2mm}

\noindent{\bf B.} Consider the function$$f(z)=2^{z}\frac{e^{2\pi i z}-1}{e^{c\cos(2\pi z)+d}},$$where $c \neq 0$ and $d \in \mathbb{C}$. It is easily verified that $\Delta_1 f$ and $f$ share any two distinct values $a, b$ CM, and that $\Delta_1 f \equiv f$. Note that $f$ can be written in the form $f(z)=2^{z}h(z)$, where$$h(z)=\frac{e^{2\pi i z}-1}{e^{c\cos(2\pi z)+d}}$$is a periodic entire function with period $1$.
\end{rem}
To continue the research, it will be interesting to answer the following question.
\begin{ques}\label{q1.1}
Can Theorem \ref{th-1.3} be generalized to transcendental meromorphic functions for more general difference operators?
\end{ques}
{The primary objective of this paper is to explore the structural properties of solutions arising from the value-sharing conditions described above. In Section 2, we provide an affirmative answer to Question 1.1 by proving several theorems that characterize the precise forms of such meromorphic functions. Specifically, Theorem \ref{th1.2} establishes the general form for transcendental meromorphic functions, while Theorems \ref{th2.2} and \ref{th1.05} refine these results for functions with \textit{Borel} and \textit{Nevanlinna exceptional values}. Through these results, we demonstrate that the relationship $L_c^n(f) \equiv Af$ leads to a specific class of periodic structures that have not been fully identified in existing literature.}\vspace{1.2mm}
\begin{rem}
	{ The answer to the Question \ref{q1.1} may be affirmative and this can be seen from the following examples.}
\end{rem}
\begin{exm}\label{Exm-1.3}
For the functions $$ f_1(z)=\displaystyle \frac{\exp\left(\frac{z \ln 2}{c}\right)\sin\left({2\pi z}/{c}\right)}{\cos \left({2\pi z}/{c}\right)-1}\; \mbox{and}\;
f_2(z)=\displaystyle \frac{\exp\left(\frac{z \ln 2}{c}\right)}{\exp\left({2\pi iz}/{c}\right)-1}, $$
where $ c\in\mathbb{C^{*}}:=\mathbb{C}\setminus\{0\} $, it is easy to verify that for $ n\in\mathbb{N} $, $ \Delta^n_{c}f_j\;(j=1, 2) $ and $ f_j $ share the values $ a $, $ b $ $ {CM} $ and $ \Delta^n_{c}f_j\equiv f_j $. More precisely, the basic fact is that if we consider the function as 
\begin{align*}
	f(z)=\exp\left(\frac{z \ln 2}{c}\right)g(z),
\end{align*} where $ g $ is a meromorphic function with $ g(z+c)=g(z) $, then it can be shown always that $ \Delta^n_{c}f\equiv f $.
\end{exm}
\begin{note}
In case of considering of meromorphic functions, the following observations can be made.\vspace{1.2mm}

\noindent{\bf a).} If $ f $ is meromorphic solution of the difference equations $ \Delta_{c}f\equiv f $, then $ f $ and $ \Delta_{c}f $ necessarily share their poles $ \infty $ $ {CM} $.
\vspace{1.2mm}

\noindent{\bf b).} The general solution of the difference equation  $ f(z+c)=2f(z) $ appeared in the conclusion of a result in \cite{Che & Yi-RM-2013} would be of the form 
\begin{align*}
	f(z)=\exp\left(\frac{z \ln 2}{c}\right)\pi(z),
\end{align*} where $ \pi(z) $ is a $ c $-periodic meromorphic function.
\end{note}
In this paper, we aim to find the general meromorphic solutions when meromorphic functions $ f $ and their higher difference operators share values. Henceforward, we define the linear difference polynomial of meromorphic functions $ f $ as (see \cite{Ahamed-IJPAM-2022,Ban & Aha & JCMA & 2020})
\begin{align}\label{e-22.3}
	L_{c}^{n}(f) = \sum_{j=0}^{n} a_j f(z+jc),
\end{align}
where $ a_n(\neq 0), \ldots, a_1, a_0\in\mathbb{C} $. It follows that that if $ a_j=\binom{n}{j}(-1)^{n-j} $, then $ L^n_c(f)=\Delta^n_cf $. Clearly, $ L^n_c(f) $ is a general setting of the difference operator $ \Delta^n_cf $.\vspace{1.2mm}
We define the class $ \mathcal{M}_c $ by 
\begin{align*}
	\mathcal{M}_c:=\{f : f\; \mbox{is meromorphic in}\; \mathbb{C}\;\mbox{and}\; f(z+c)=f(z)\}.
\end{align*}
It can be shown that $ \mathcal{M}_c $ is a field of meromorphic functions of period $ c $. However, a little is know about the precise form of the meromorphic solutions of difference equations when meromorphic functions and their difference polynomials share values. \vspace{1.2mm}

{ The formulation \eqref{e-22.3} significantly generalizes the classical $n$-th order forward difference operator $\Delta_c^n f$, which is recovered when the coefficients $a_j$ are chosen as binomial coefficients $(-1)^{n-j} \binom{n}{j}$. While much of the existing literature (see e.g., \cite{Mallick-AAMP-2022,Che & Yi-RM-2013,Hal & Kor - JPA-2007}) focuses on the specific properties of $\Delta_c^n f$, our results for $L_c^n(f)$ provide a broader framework that encompasses a wider class of functional transformations. By moving from this specific case to a general linear operator with arbitrary constant coefficients, we provide a more robust characterization of the value distribution and the structural form of meromorphic solutions.}
\section{\bf Main results}
For a meromorphic function $ f $, considering a general setting $ L_{\eta}(f):=a_1f(z+c)+a_0f(z) $ of the difference operator $ \Delta_{\eta}f $, Ahamed \cite{Aha & RM & 2019} has improved the result of Chen and Yi \cite{Che & Yi-RM-2013} in view of sharing values. We define linear differential polynomial of $ f $ as 
\begin{align*}
	L_k[f]=b_kf^{(k)}+\cdots+b_1f^{\prime}+b_0, \mbox{where}\; b_k(\neq 0), \ldots, b_1, b_0\in\mathbb{C}.
\end{align*}
As an application of logarithmic derivative lemma and its difference analogue results, uniqueness results of certain difference-differential polynomials of meromorphic functions can be studied. We prove the following uniqueness result.
\begin{thm}\label{th1.2}
Let $ f $ be a transcendental meromorphic function such that its order of growth $ \sigma(f) $ is not an integer or infinite, and let $ \eta\in\mathbb{C} $ be a constant such that $ f(z+c)\not\equiv f(z) $. If ${L}^n_{c}(f)+L_k[f] $ and $ f $ share three distinct values $ a, b, \infty $ $ CM $, then $ {L}^n_{c}(f)+L_k[f] \equiv f $.
\end{thm}
We have the following immediate corollaries from Theorem \ref{th1.2}.
\begin{cor}
Let $ f $ be a transcendental meromorphic function such that its order of growth $ \sigma(f) $ is not an integer or infinite, and let $ \eta\in\mathbb{C} $ be a constant such that $ f(z+c)\not\equiv f(z) $. If ${L}^n_{c}(f) $ and $ f $ share three distinct values $ a, b, \infty $ $ CM $, then $ {L}^n_{c}(f)\equiv f $. Furthermore, \vspace{1.2mm}

\noindent{\bf A.} $ 	f(z)=\displaystyle\rho_1^{z/c}\pi_1(z)+\rho_2^{z/c}\pi_2(z)+\cdots+\rho_n^{z/c}\pi_n(z), $
where $ \pi_j(z)\in\mathcal{M}_c\; (j=1, 2, \ldots, n) $ and $ \rho_j $ $ (j=1, 2, \ldots, n) $ are distinct roots of the equation $ a_nw^n+\cdots+a_1w+a_0-1=0. $\vspace{1.2mm}
In particular, if $ \rho_j\in \{0, 1\} $  be such that at least one of $ \rho_j $'s are non-zero, then $ f\in\mathcal{M}_c $.\vspace{1.2mm}

\noindent{\bf B.} \begin{align*}
	f(z)=\left(\sum_{m_1=1}^{N_1-1}z^{m_1}\pi_{m_1}(z)\right)\sigma_1^{z/c}+\cdots+\left(\sum_{m_q=1}^{N_q-1}z^{m_q}\pi_{m_q}(z)\right)\sigma_q^{z/c},
\end{align*}
where $ \pi_j\in\mathcal{M}_c $ and $ \sigma_1, \ldots $, $ \sigma_q $ are multiple roots, of respective multiplicities $ N_1, \ldots, N_q $, of the equation $ a_nw^n+\cdots+a_1w+a_0-1=0. $
\end{cor}
\begin{cor}
Let $ f $ be a transcendental meromorphic function such that its order of growth $ \sigma(f) $ is not an integer or infinite. If $L_k[f] $ and $ f $ share three distinct values $ a, b, \infty $ $ CM $, then $ L_k[f]\equiv f $. Furthermore, the function $ f $ takes the form $ f(z)=c_1e^{\lambda_1 z}+\cdots+c_ne^{\lambda_n z}, $ where $ \lambda_1, \ldots, \lambda_n $ are distict roots of the equation $ b_kw^k+\cdots+b_1w+b_0=0 $.
\end{cor}
We obtain the next result in case of when the function $ f $ has a `Nevanlinna exceptional value'.
\begin{thm}\label{th2.2}
	{ Let $f$ be a finite order transcendental entire function with a finite Nevanlinna exceptional value $\zeta$, and let $\eta \in \mathbb{C}$ such that $f(z+\eta) \not\equiv f(z)$. If $L_c^n(f)$ and $f$ share the value $a$ CM, then $a = 0$ and $L_c^n(f) \equiv Bf$ for some non-zero constant $B$. Furthermore, the function $f(z)$ must belong to one of the following structural classes based on the roots of the characteristic equation $P(w) = \sum_{j=0}^n a_j w^j - B = 0$.\vspace{1.2mm}
		
		\noindent{Class I:} {Distinct Roots (The linear combination class).} If the characteristic equation has $n$ distinct roots $\rho_1, \rho_2, \ldots, \rho_n$, then $f(z)$ is a combination of quasi-periodic functions:$$f(z) = \sum_{j=1}^{n} \rho_j^{z/c} \pi_j(z)$$where each $\pi_j(z) \in \mathcal{M}_c$ is a periodic meromorphic function with period $c$.\vspace{1.2mm}
		
		\noindent Note: If any $\rho_j \in \{0, 1\}$ (with at least one non-zero), then $f(z)$ reduces to a periodic function in $\mathcal{M}_c$.\vspace{2mm}
		
		\noindent{Class II: Multiple roots (The polynomial-quasi-periodic class.)} If the characteristic equation has multiple roots $\sigma_1, \sigma_2, \ldots, \sigma_q$ with multiplicities $N_1, N_2, \ldots, N_q$ respectively, then $f(z)$ takes the expanded form:$$f(z) = \sum_{k=1}^{q} \left( \sum_{m_k=1}^{N_k-1} z^{m_k} \pi_{m_k}(z) \right) \sigma_k^{z/c}$$where $\pi_j \in \mathcal{M}_c$.}
\end{thm}
We obtain the following corollary as a consequence of Theorem \ref{th2.2}
\begin{cor}\label{cor-2.2}
Let $ f $ be a finite order transcendental entire function which has a finite Nevanlinna exceptional value $ \zeta $ and $ \eta\in\mathbb{C} $ be a constant such that $ f(z+\eta)\not\equiv f(z) $. If $ {L}^2_{c}(f) $ and $ f $ share $ a $ $ CM $, then $ 	a=0\;\;\text{and}\;\; {L}_{c}^2(f)\equiv {B}f, $
	 where $ {B} $ is a non-zero constant. Furthermore, the function assume one of the following forms:
	\begin{enumerate}
		\item[(i)] If $ a_1^2+4a_2{B}\neq 4a_2a_0 $, then  $ f(z) $ must of the form 
		\begin{equation*}
			f(z)=[{R}_1({B})]^{z/c}\pi_1(z)+\zeta \;\; \text{or} \;\; f(z)=[{R}_2({B})]^{z/c}\pi_2(z)+\zeta,
		\end{equation*}
		 where  $ \pi_j(z)\in \mathcal{M}_c\;(j=1, 2) $ with $ \max\{\sigma(\pi_1),\;\sigma(\pi_2)\}\leq \sigma(f).  $
		\item[(ii)] If $ a_1^2+4a_2{B}= 4a_2a_0 $, then $ f(z)=\left(-\frac{a_1}{2a_2}\right)^{z/c}\pi(z)+\zeta, $ where $ \pi(z)\in\mathcal{M}_c $ with $ \sigma(\pi)\leq \sigma(f) $.
	\end{enumerate}
\end{cor}
Text result we obtained in this paper is the following when $ f $ has a finite Borel exceptional value.
\begin{thm}\label{th1.05}
	{ Let $f$ be a finite order transcendental entire function with a finite Borel exceptional value $\zeta$. Let $L_\eta^n(f)$ be the linear difference operator defined in \eqref{e-22.3} such that $\sum_{j=0}^{n}a_j=0$ and $\eta \in \mathbb{C} \setminus \{0\}$ with $f(z+c) \not\equiv f(z)$. If $L_c^n(f)$ and $f$ share the value $d \neq \zeta$ CM, then $d = 0$ and $L_\eta^n(f) \equiv Af$ for some non-zero constant $A$. The structural characterization of the solution $f(z)$ is determined by the roots of the characteristic equation $P(w) = \sum_{j=0}^{n} a_j w^j - A = 0$ as follows:\vspace{2mm}
		
		\noindent Class I: Distinct Roots (Quasi-Periodic Representation). If the characteristic equation $P(w) = 0$ possesses $n$ distinct roots $\rho_1, \rho_2, \ldots, \rho_n$, then $f(z)$ is expressed as:$$f(z) = \sum_{j=1}^{n} \rho_j^{z/c} \pi_j(z),$$ where $\pi_j(z) \in \mathcal{M}_c$ are periodic meromorphic functions with period $c$.\vspace{1.2mm} 
		
		\noindent Special Case: If $\rho_j \in \{0, 1\}$ (with at least one root non-zero), the solution $f(z)$ reduces to a periodic function in $\mathcal{M}_c$.\vspace{1.2mm}
		
		\noindent Class II: Multiple Roots (General Structural Form). If the characteristic equation $P(w) = 0$ has multiple roots $\sigma_1, \sigma_2, \ldots, \sigma_q$ with multiplicities $N_1, N_2, \ldots, N_q$ respectively, then $f(z)$ admits the representation:$$f(z) = \sum_{k=1}^{q} \left( \sum_{m_k=1}^{N_k-1} z^{m_k} \pi_{m_k}(z) \right) \sigma_k^{z/c},$$ where $\pi_j \in \mathcal{M}_c$.}
\end{thm}
In particular for $ n=2 $, we obtain the following result as an immediate corollary of Theorem \ref{th1.05}.
\begin{cor}\label{cor-1.05}
	Let $ f $ be a finite order transcendental entire function of finite order having a finite Borel exceptional value $ \zeta $ and $ {L}^2_{\eta}(f) $ be defined in \eqref{e-22.3} such that $ a_2+a_1+a_0=0 $ and $ \eta\in\mathbb{C}\setminus\{0\} $ be a constant such that $ f(z+c)\not\equiv f(z) $. If $ {L}^2_{c}(f) $ and $ f $ share $ d \;(\neq \zeta) $ $ CM $, then $ d=0\;\;\text{and}\;\; {L}^n_{\eta}(f)\equiv {A}f, $
	where $ {A}\in \mathbb{C}\setminus\{0\} $. Furthermore,
	\begin{enumerate}
		\item[(i)] If $ a_1^2+4a_2{A}\neq 4a_2a_0 $, then  $ f(z) $ must of the form 
		\begin{equation*}
			f(z)=[{R}_1({A})]^{z/c}\pi_1(z)\;\; \text{or} \;\; f(z)=[{R}_2({A})]^{z/c}\pi_2(z), 
		\end{equation*}
		where  $ \pi_j(z)\in\mathcal{M}_c\; (j=1, 2) $ with $ 	\max\{\sigma(\pi_1),\;\sigma(\pi_2)\}\leq \sigma(f). $
		\item[(ii)] If $ a_1^2+4a_2{A}= 4a_2a_0 $, then $ f(z)=\left(-\frac{a_1}{2a_2}\right)^{z/c}\pi(z), $
		where $ \pi(z)in\mathcal{M}_c $ with $ \sigma(\pi)\leq \sigma(f) $.
	\end{enumerate}
\end{cor}
Following is a corollary of {Theorem \ref{th1.05}}.
\begin{cor}
	Under the assumption of {Theorem \ref{th1.05}}, if $ f $ has a non-zero finite Borel exceptional value $ a $, then for $ \eta\neq 0 $, the value $ a $ is not shared by the functions $ {L}_{\eta}^n(f) $ and $ f $. 
\end{cor}
\begin{exm}
	Let $ f(z)=be^z+a $, where $ a, b $ are two distinct complex numbers. We see that $ f $ has the Borel exceptional value $ a $. Clearly, for any $ c\neq 2k\pi i $, $ k\in\mathbb{Z} $, and for the choices of the constants $ a_0, a_1 $ and $ a_2 $ satisfying $ a_0+a_1+\cdots+a_n=0 $,  one can verify that the value $ a $ is not shared by ${L}_{c}^n(f)=b\left(a_ne^{n c}+\cdots+a_1e^{c}+a_0\right)e^z $ and $ f(z) $.
\end{exm}
\section{\bf Some Preliminary results}
\par In this section, we present some necessary lemmas which will play key role to prove the main results of this paper. 
\begin{lem}\cite{Chi & Fen-RM-2008}\label{lem2.1}
	Let $ f $ be a meromorphic function with a finite order $ \sigma $, and $ \eta $ be a non-zero constant. Then for any $ \epsilon >0 $, we have \begin{equation}
		\label{e2.11} m\left(r,\frac{f(z+\eta)}{f(z)}\right)=O\left(r^{\sigma-1+\epsilon}\right).
	\end{equation}
\noindent The equation (\ref{e2.11}) in which $ \sigma $ is the (finite) order of $ f $, and $ \epsilon >0 $, implies
\begin{equation*}
	m\left(r,\frac{f(z+\eta)}{f(z)}\right)=S(r,f),
\end{equation*}
 possibly outside a set of finite logarithmic measure.
\end{lem}
\begin{lem}\cite{Moh-FFA-1971,Val-BSM-59}\label{lem2.2}
	If $ \mathcal{R}(f) $ is rational in $ f $ and has small meromorphic
	coefficients, then 
	\begin{equation*}
		T(r,\mathcal{R}(f))=\deg_f(\mathcal{R})T(r,f)+S(r,f).
	\end{equation*}
\end{lem}

\begin{lem}\cite{Chi & Fen-RM-2008}\label{lem2.3}
	Let $ f $ be a transcendental meromorphic function with finite order $ \sigma $, and $ \eta $ be a non-zero constant. Then for each $ \epsilon >0 $, we have 
\begin{align*}
	T(r,f(z+\eta))&=T(r,f)+O\left(r^{\sigma-1+\epsilon}\right)+O\left(\log r\right)\\T(r,f(z+\eta))&=T(r,f)+S(r,f),
\end{align*}
	 possibly outside of a finite logarithmic measure.
\end{lem}
We establish the following lemma, which will be instrumental in proving the main result of this paper.
\begin{lem}\label{lem2.5}
	Let $ f $ be a transcendental meromorphic function of finite order. Then \begin{equation*}
		m\left(r,\frac{f^{(k)}(z+\eta_1)}{f(z+\eta_2)}\right)=S(r,f),
	\end{equation*} for all $ z $ satisfying $ |z|=r\not\in E, $ where $ E $ is a set with finite logarithmic measure, and $ \eta_1 $, $ \eta_2 $ are constants and $ k $ is a non-negative integer. 
\end{lem}
\begin{proof} Using {Lemma \ref{lem2.1}}, we obtain 
\begin{align*}
m\left(r,\frac{f^{(k)}(z+\eta_1)}{f(z+\eta_2)}\right)&= m\left(r,\frac{f^{(k)}(z+\eta_1)}{f(z+\eta_1)}\frac{f(z+\eta_1)}{f(z)}\frac{f(z)}{f(z+\eta_2)}\right) \\&\leq m\left(r,\frac{f^{(k)}(z+\eta_1)}{f(z+\eta_1)}\right)+m\left(r,\frac{f(z+\eta_1)}{f(z)}\right)+m\left(r,\frac{f(z)}{f(z+\eta_2)}\right)\\&\quad\quad+O(1)\\&= S(r,f).
\end{align*} This completes the proof.
\end{proof}	
\begin{lem}\label{lem22.6}\cite{Chi & Fen-RM-2008}
Let $ \eta_1 $ and $ \eta_2 $ be two arbitrary complex numbers, and let $ f(z) $ be a meromorphic function of finite order $ \sigma $. Let $ \epsilon>0 $ be given, then there exists a subset $ E\subset (0,\infty) $ with finite logarithmic measure such that for all $ |z|=r\not\in E\cup [0,1], $ we have 
	\begin{equation*}
\exp\bigg(-r^{\sigma-1+\epsilon}\bigg)\leq \bigg|\frac{f(z+\eta_1)}{f(z+\eta_2)}\bigg|\leq \exp\bigg(r^{\sigma-1+\epsilon}\bigg).
	\end{equation*}
\end{lem}
\begin{rem}
In { Lemma \ref{lem22.6}}, if $ \sigma <1 $, let $ \epsilon={(1-\sigma)}/{2}>0 $, then we can see that 
\begin{align*}
	\frac{f(z+\eta_1)}{f(z+\eta_2)}\rightarrow 1\; \mbox{as}\;|z|=r\not\in E\cup [0,1],\; r\rightarrow\infty.
\end{align*}
\end{rem}
We prove the following lemma, which serves as a key component in the proof of our main theorem.
\begin{lem}\label{lem2.6}
	Let $ f $ be a non-constant meromorphic function and $ {L}^n_{c}(f) $ be defined in \eqref{e-22.3}. If $ f $ and $ {L}^n_{c}(f) $ share the values $ a $, $ b $ and $ \infty $ $ {CM} $, then $ f\equiv L^n_{\eta}(f) $ and $ f $ is not a rational function. 
\end{lem} 
\begin{proof}
	Let $g := L^n_{c}(f)$. Without loss of generality, assume that $a=0$. Since $f$ and $g$ share the values $0$ and $ \infty $, it follows from Lemma 2.1 in \cite[p. 110]{Yan & Yi-2006} that there exists a non-zero constant $K$ such that $f=Kg$. Furthermore, because $f$ and $g$ share the value $ b $CM, there exists a point$ z_0 \in \mathbb{C} $ such that $f(z_0) = g(z_0) = b$. This implies that $ K=1 $, and we thus conclude that $f \equiv g$.\vspace{1.5mm}
	
	\par We proceed by contradiction to prove that $f$ is not a rational function. Suppose, to the contrary, that $f$ is rational. Then $f$ can be written as $f(z) = P(z)/Q(z)$, where $P$ and $Q$ are coprime polynomials such that $PQ \not\equiv 0$. Let $E(0, P) = \{z : P(z) = 0\}$ and $E(0, Q) = \{z : Q(z) = 0\}$ denote the sets of zeros of $P$ and $Q$, respectively. It follows immediately that
	\begin{equation}
		\label{e1.1} E(0, P) \cap E(0, Q) = \emptyset.
	\end{equation}
	A simple calculation shows that
\begin{align*}
a_nf(z+nc)&+\cdots+a_1f(z+c)+a_0f(z)\\&=a_n\frac{P(z+nc)}{Q(z+nc)}+\cdots+a_1\frac{P(z+c)}{Q(z+c)}+a_0\frac{P(z)}{Q(z)}\\&=\frac{M_n(z)+\cdots+M_1(z)+M_0(z)}{Q(z+nc)\cdots Q(z+c)Q(z)}=\frac{P_1(z)}{Q_{1}(z)},
\text{(say)}
\end{align*}
where \[M_k(z):=a_kP(z+kc)\prod_{j=0, j\neq k}^{n}Q(z+j\eta),\; k=0, 1, 2, \ldots, n,\]
 $ P_1(z):=M_n(z)+\cdots+M_1(z)+M_0(z) $ and $ Q_1(z):=Q(z+nc)\cdots Q(z+\eta)Q(z) $ are two relatively prime polynomials and $ P_1(z)Q_1(z)\not\equiv 0 $. Furthermore, $ f $ is a rational function and satisfying $ E(a,f) = E(a,a_nf(z+nc)+\cdots+a_1f(z+c)+a_0f(z)) $, hence, there exists a polynomial $ h(z) $ such that 
	\begin{equation*}
			a_nf(z+nc)+\cdots+a_1f(z+c)+a_0f(z)-a=(f-a)h(z).
	\end{equation*} 
This can be expressed as
\begin{align*}
	\frac{M_n(z)+\cdots+M_1(z)+M_0(z)}{Q(z+nc)\cdots Q(z+c)Q(z)}-a\equiv
	\left(\frac{P(z)}{Q(z)}-a\right)h(z).
\end{align*}
 \noindent{\bf{Case 1.}} Let $P(z)$ be a non-constant polynomial. By the Fundamental Theorem of Algebra, there exists at least one $z_0 \in \mathbb{C}$ such that $P(z_0) = 0$. Consequently, we have 
\begin{equation}
\label{e2.4}
a_n\frac{P(z_0+nc)}{Q(z_0+nc)}+\cdots+a_1\frac{P(z_0+c)}{Q(z_0+c)}\equiv (1-h(z_0))a.
\end{equation}
\noindent{\bf{Subcase 1.1.}} Let $a=0$. It follows from (\ref{e2.4}) that $P(z_0+kc)=0$ for each $k=1, 2, \dots, n$, given that $a_1, a_2, \dots, a_n$ are arbitrary complex constants. From (\ref{e1.1}), we deduce by induction that $P(z_0+mc)=0$ for all positive integers $m$. Since a non-zero polynomial cannot have infinitely many zeros, $P(z)$ must be a non-zero constant.\\
	
\noindent{\bf{Subcase 1.2.}} We suppose that $ a\neq 0 $. From (\ref{e2.4}), it follows that
\begin{equation*}
P_1(z_0)=M_n(z_0)+\cdots+M_1(z_0)\equiv(1-h(z_0))aQ(z_0+nc)\cdots Q(z_0+c)=(1-h(z_0))aQ_1(z_0).
\end{equation*}
By the same argument used in the previous case, it can be shown that
\begin{align}\label{e22.5}P_1(z_0+mc)=(1-h(z_0))aQ_1(z_0+mc)
\end{align}
From (\ref{e2.4}) and (\ref{e22.5}), we obtain 
\begin{equation*}
\frac{P_1(z_0)}{Q_1(z_0)}=\frac{P_1(z_0+mc)}{Q_1(z_0+mc)}\; \text{for all	positive integer $ m $}
\end{equation*}
 which contradicts the fact that $ E(0,P_1)\cap
E(0,Q_1)=\phi $. Therefore, it follows that $ f(z) $ takes the form $ f(z)={d}/{Q(z)} $, where $ P(z)=d=\text{constant}\; (\neq
	0).$\vspace{1.5mm}
	
	\noindent{\bf{Case 2.}} Let $ Q(z) $ be a non-zero constant. A simple computation shows that
	\begin{align}
		\label{e2.5} a_nf(z+nc)+\cdots+a_1f(z+c)+a_0f(z)=\frac{d\bigg(N_n(z)+\cdots+N_1(z)+N_0(z)\bigg)}{Q(z)Q(z+c)\cdots Q(z+nc)},
	\end{align} 
where \[N_k(z):=a_k\prod_{j=0, j\neq k}^{n}Q(z+jc),\; k=0, 1, 2, \ldots, n.\] Since $ E(b,f)=E(b,a_nf(z+nc)+\cdots+a_1f(z+c)+a_0f(z)) $, then there exists a polynomial $ h_1(z) $ such that $
	a_nf(z+nc)+\cdot+a_1f(z+c)+a_0f(z)-b=(f-b)h_1(z) $ which can be written as	
	\begin{align}
		\label{e2.7} N_n(z)+\cdots+N_0(z)-\frac{b}{d}\equiv \bigg(\frac{d-bQ(z)}{d}\bigg)h_1(z)Q(z+nc)\cdots Q(z+c).
	\end{align} Since $ Q(z) $, $Q(z+c) $, $ \ldots $ ,$ Q(z+nc) $ all are non-constant polynomials, employing {Fundamental Theorem of Algebra}, there exist $ z_0 $, $ z_1 $, $ \ldots, $ $ z_n $ such that $ Q(z_0)=0, $ $ Q(z_1+c)=0 $, $ \ldots $, $ Q(z_n+nc)=0 $.\vspace{1.2mm} 

If $Q(z_0)=0$, then (\ref{e2.7}) yields $h_1(z_0)=a_0 - b/d$, a contradiction. Similarly, if $Q(z_1+\eta)=0$, then (\ref{e2.7}) implies that $Q(z_1+kc)=0$ for some $k \in \{0, 2, \dots, n\}$, which is impossible. If $Q(z_2+2c)=0$, then (\ref{e2.7}) again implies that $Q(z_2+kc)=0$ for some $k \in \{0, 1, \dots, n\}$, which is likewise impossible. Proceeding inductively, we arrive at a contradiction in all cases. Consequently, the polynomial $Q(z)$ must be a non-zero constant, say $d_2$. It then follows that $f(z) = d/d_2$ is a constant, contradicting the initial assumption that $f$ is non-constant. This completes the proof.
\end{proof}
By Lemma \ref{lem2.6}, any meromorphic solution $f$ to the difference equation $L^n_c(f) = f$ must be transcendental. The following result provides a complete characterization of the precise functional form of these solutions. For a comprehensive derivation of these forms, we refer the reader to the methods established in \cite[Eqs. (4.19)–(4.24)]{Ahamed-IJPAM-2022}.
\begin{lem}\cite{Ahamed-IJPAM-2022}\label{lem-3.7}
	Let $ f $ be a non-constant meromorphic of finite order and $ c $ be a constant such that $ f(z+c)\not\equiv f(z) $. If $ L^n_c(f)=f $, then 
	\begin{enumerate}
		\item[(i)] $ 	f(z)=\displaystyle\rho_1^{z/c}\pi_1(z)+\rho_2^{z/c}\pi_2(z)+\cdots+\rho_n^{z/c}\pi_n(z), $
		where $ \pi_j(z)\in\mathcal{M}_c\; (j=1, 2, \ldots, n) $ and $ \rho_j $ $ (j=1, 2, \ldots, n) $ are distinct roots of the equation $ a_nw^n+\cdots+a_1w+a_0-1=0. $ In particular, if $ \rho_j\in \{0, 1\} $  be such that at least one of $ \rho_j $'s are non-zero, then $ f\in\mathcal{M}_c $. 
		\item[(ii)] 
		\begin{align*}
			f(z)=\left(\sum_{m_1=1}^{N_1-1}z^{m_1}\pi_{m_1}(z)\right)\sigma_1^{z/c}+\cdots+\left(\sum_{m_q=1}^{N_q-1}z^{m_q}\pi_{m_q}(z)\right)\sigma_q^{z/c},
		\end{align*}
		where $ \pi_j\in\mathcal{M}_c $ and $ \sigma_1, \sigma_2 $, $ \ldots $, $ \sigma_q $ are multiple roots, of respective multiplicities $ N_1, N_2, \ldots, N_q $, of the equation $ a_nw^n+\cdots+a_1w+a_0-1=0. $ 
	\end{enumerate}
\end{lem}

\begin{lem}\cite{Yan & Yi-2006}\label{lem2.8}
	Suppose that $ n\geq 2 $ and let $ f_1,\ldots,f_n $ be meromorphic functions and $ g_1,\ldots,g_n $ be entire functions such that \begin{enumerate}
		\item[(i)] $ \displaystyle\sum_{j=1}^{n}f_j\exp(g_j)\equiv 0 $;
		\item[(ii)] when $ 1\leq j<k\leq n $, $ g_j-g_k $ is not constant;\vspace{1.5mm}
		\item[(iii)] when $ 1\leq j\leq n $, $ 1\leq s<k\leq n $, 
		\begin{equation*}
			T(r,f_j)=o\{T(r,\exp\{g_s-g_k\})\} \; (r\rightarrow\infty, r\not\in E),
		\end{equation*}
		 where $ E\subset (1,\infty) $ has finite linear measure or finite logarithmic measure. Then $ f\equiv 0 $, $ j=1, 2, \ldots, n. $ 
	\end{enumerate}
\end{lem}
\section{\bf Proof of the main results}

\begin{proof}[\bf Proof of Theorem \ref{th1.2}]
	Since $ f $ is a finite order meromorphic function and $ {L}_{c}^n(f)+L_k[f] $ and $ f $ share the values $ a $ and $ \infty $ $ CM $, then we have 
	\begin{equation}
		\label{e4.1} \frac{{L}_{c}^n(f)+L_k[f]-a}{f(z)-a}=e^{\alpha(z)}, 
	\end{equation}
	where $ \alpha(z) $ is a polynomial, with 
	\begin{equation}
		\label{e4.2} \deg(\alpha)\leq\sigma(f).
	\end{equation}
	In fact, $ \deg(\alpha)<\sigma(f) $ since $ \sigma(f) $ is not an integer. It follows from (\ref{e4.1}) that
	\begin{equation}
		\label{e4.3} a_nf(z+nc)+\cdots+a_1f(z+c)+L_k[f]+\left(a_0-e^{\alpha(z)}\right)f(z)=a\left(1-e^{\alpha(z)}\right).
	\end{equation}
	\noindent Similarly, $ b $ and $ \infty $ $ CM $ being shared by $ {L}_{c}(f) $ and $ f $, we also have 
	\begin{equation}
		\label{e4.4} a_nf(z+nc)+\cdots+a_1f(z+c)+L_k[f]+\left(a_0-e^{\beta(z)}\right)f(z)=b\left(1-e^{\beta(z)}\right), 
	\end{equation}
	where $ \beta(z) $ is a polynomial and 
	\begin{equation}
		\label{e4.5} \deg(\beta)<\sigma(f). 
	\end{equation}
	\noindent Combining (\ref{e4.3}) and (\ref{e4.4}), we obtain
	\begin{equation*}
		\bigg(\left(a_0-e^{\alpha(z)}\right)-\left(a_0-e^{\beta(z)}\right)\bigg)f(z)= a\left(1-e^{\alpha(z)}\right)-b\left(1-e^{\beta(z)}\right)
	\end{equation*}
	which implies that 
	\begin{equation}
		\label{e4.6}  \bigg(e^{\beta(z)}-e^{\alpha(z)}\bigg)f(z)= a\left(1-e^{\alpha(z)}\right)-b\left(1-e^{\beta(z)}\right).
	\end{equation}
	\par We claim that $e^{\beta(z)} - e^{\alpha(z)} \equiv 0$. To the contrary, suppose that $e^{\beta(z)} - e^{\alpha(z)} \not\equiv 0$. Based on (\ref{e4.2}) and (\ref{e4.6}), it is clear that the order of growth of the left-hand side (L.H.S.) of (\ref{e4.6}) is $\sigma(f)$, whereas the order of growth of the right-hand side (R.H.S.) is strictly less than $\sigma(f)$. This yields a contradiction. Therefore, we must have $e^{\alpha(z)} \equiv e^{\beta(z)}$. Consequently, from (\ref{e4.6}), we can easily obtain
	\begin{equation}
		\label{e4.7} (a-b)\left(1-e^{\beta(z)}\right)=0.
	\end{equation}
	Since $a \neq b$, it follows from (\ref{e4.7}) that
	\begin{equation}
		\label{e4.8} e^{\beta(z)}=e^{\alpha(z)}=1.
	\end{equation}
	\par Therefore, by combining (\ref{e4.1}) and (\ref{e4.8}), we obtain $L^n_{\eta}(f) + L_k[f] \equiv f$. This completes the proof.
\end{proof}	
\begin{proof}[\bf Proof of Theorem \ref{th2.2}]
Since $ a $ is the finite Borel exceptional value of $ f $, by Hadamard's
factorization theory, $ f(z) $ can be expressed as
\begin{equation}
	\label{e3.1} f(z)={H}(z)e^{h(z)}+a ,
\end{equation}
where $ {H}(\not\equiv 0) $ is an entire function, $ h(z) $ is a polynomial, and $ {H} $ and $ h $ satisfy $ \lambda({H})=\sigma({H})=\lambda(f-a)<\sigma(f)=\deg(h) .$ Since $ {L}_{c}^n(f) $ and $ f $ share the value $ a $ $ CM $, then we have
\begin{align}
	\label{e3.2} \frac{{L}_{c}^n(f)-a}{f-a}&= \frac{a_n{H}(z+nc)e^{h(z+nc)}+\cdots+a_1{H}(z+c)e^{h(z+c)}+a_0f(z)-a}{f-a}\\&=\nonumber e^{p(z)},
\end{align}
where $ p(z) $ is a polynomial in $ z $. Set
\begin{eqnarray}
	\label{e3.3} h(z):=\alpha_kz^k+\cdots+\alpha_0\;\; \text{and}\;\; p(z)=b_lz^l+\ldots+b_0,
\end{eqnarray}
where $ k=\sigma(f) $, $ l=\deg p $, $ \alpha_k(\neq 0), \ldots, \alpha_0 $, $ b_l(\neq 0), \ldots, b_0 $ are all constants.\vspace{1.5mm}

\noindent From (\ref{e3.2}), it follows that $\deg p \leq \deg h$. To complete the proof, it remains to establish two points: first, that $p(z)$ is a constant, and second, that $a = 0$..\vspace{1.5mm}

\noindent {\bfseries{Step I:}} We show that $ p(z) $ is constant, \emph{i.e.,} $ l=0 $.\vspace{1.5mm}

\noindent{\bfseries{Case 1.}} We suppose that $ 1\leq l<k $. If $ a\neq 0 $, then from (\ref{e3.2}), we obtain
\begin{align}
	\label{e3.4} & a_n{H}(z+nc)e^{h(z+nc)-h(z)}+\cdots+a_1{H}(z+c)e^{h(z+c)-h(z)}+a_0{H}(z)-{H}(z)e^{p(z)}\\&=\nonumber -2ae^{-h(z)}. 
\end{align}
A simple computation shows that
\begin{align*}
	 \deg\left[h(z+jc)-h(z)\right]&=k-1\; \mbox{for}\; j=1, 2, \ldots, n,\\ \deg h(z)&=k,\; \deg p(z)=l<k,
\end{align*}
 \par Therefore, we obtain $\sigma(H) < k$. It follows that the order of growth of the left-hand side (L.H.S.) of (\ref{e3.4}) is less than $k$, whereas the order of the right-hand side (R.H.S.) is exactly $k$. This contradiction implies that $a=0$. Consequently, equation (\ref{e3.4}) reduces to
\begin{equation}
	\label{e3.5} a_n\frac{{H}(z+nc)}{{H}(z)}e^{h(z+nc)-h(z)}+\cdots+a_1\frac{{H}(z+c)}{{H}(z)}e^{h(z+c)-h(z)}+a_0 = e^{p(z)}. 
\end{equation}
\par It follows from (\ref{e3.5}) that $H(z+nc)/H(z), \ldots, H(z+c)/H(z)$ are all non-zero entire functions. Let $\sigma(H) = \sigma_1$. It is clear that $\sigma_1 < \sigma(f) = k$. By Lemma \ref{lem22.6}, for any given $\epsilon$ satisfying $0 < 3\epsilon < k - \sigma_1$, there exists a set $E \subset (1, \infty)$ of finite logarithmic measure such that for all $z$ with $|z| = r \notin [0, 1] \cup E$,
\begin{align}
	\label{e3.6} \exp\bigg(-r^{\sigma_1-1+\epsilon}\bigg)\leq \bigg|\frac{f(z+jc)}{f(z)}\bigg|\leq \exp\bigg(r^{\sigma_1-1+\epsilon}\bigg)\; \mbox{for}\; j=1, 2, \ldots, n. 
\end{align}
Since the functions $H(z+jc)/H(z)$ are entire for $j=1, 2, \ldots, n$, it follows from (\ref{e3.6}) that
\begin{align*}
	T\left(r, \frac{H(z+jc)}{H(z)} \right) = m\left(r, \frac{H(z+jc)}{H(z)} \right) \leq r^{\sigma_1-1+\epsilon}.
\end{align*}
This implies that $\sigma(H(z+jc)/H(z)) \leq \sigma_1-1+\epsilon < k-1$. Given $s < k$, it follows that $\deg p(z) \leq k-1$. If we assume $\deg p(z) < k-1$, then clearly $\deg(h(z+jc)-h(z)) = k-1$ for all $j=1, 2, \ldots, n$. Consequently, the order of growth of the left-hand side (L.H.S.) of (\ref{e3.5}) is $k-1$, whereas the growth of the right-hand side (R.H.S.) is $\deg p(z) < k-1$. This discrepancy yields a contradiction.\vspace{1.5mm}

\par If $\deg p(z) = k-1$, then since $H(z+jc)/H(z)$ are entire and $\deg(h(z+jc)-h(z)) = k-1$, it follows that the functions $(H(z+jc)/H(z))e^{h(z+jc)-h(z)}$ have $0$ as a Borel exceptional value for each $j=1, \dots, n$. However, based on (\ref{e3.5}), it can be seen that $0$ is not a Borel exceptional value of the left-hand side (L.H.S.), whereas it is necessarily a Borel exceptional value of the right-hand side (R.H.S.). This leads to a contradiction.\vspace{1.5mm}

\noindent{\bfseries{Case 2.}} Suppose that $1 \leq l \leq k$. Regarding $p$ and $h$, we distinguish the following three cases: (i) $b_k = a_k$; (ii) $b_k = -a_k$; and (iii) $b_k \neq \pm a_k$.\vspace{1.5mm}

\noindent{\bfseries{Subcase 2.1.}} Let $ b_k=a_k $. First we suppose that $ a\neq 0 $. Thus (\ref{e3.2}) is expressed as 
\begin{equation*}
	a_n{H}(z+nc)e^{h(z+nc)}+\cdots+a_1{H}(z+c)e^{h(z+c)}+a_0{H}(z)e^{h(z)}+2a={H}(z)e^{h(z)}e^{p(z)}, 
\end{equation*}
and this can be re-written as 
\begin{equation}
	\label{e3.8} g_1(z)e^{p(z)}+g_2(z)e^{-p(z)}+g_3(z)e^{h_0(z)}\equiv 0, 
\end{equation}
where 
\[
\begin{cases}
	g_1(z)=-{H}(z)\\
	g_2(z)=2a\\
	g_3(z)=a_n{H}(z+2c)e^{h(z+nc)-h(z)}+\cdots+a_1{H}(z+c)e^{h(z+c)-h(z)}+a_0{H}(z)\\ h_0(z)=0
\end{cases}
\]
\par We observe that $\deg\{h(z+tc)-h(z)\} = k-1$ for $t=1, 2, \ldots, n$. Assuming $\sigma(H) < k$, it is easily verified that $\sigma(g_j) < k$ for $j=1, 2, 3$. Since $b_k = a_k$, it follows that $\deg(p+h) = k$, $\deg(p-h_0) = k$, and $\deg(h-h_0) = k$. Given that the functions $e^{p+h}$, $e^{p-h_0}$, and $e^{-h-h_0}$ are of regular growth with order $k$, and since $\sigma(g_j) < k$ for $j=1, 2, 3$, a simple computation shows that
\begin{align}	\label{e3.9}
	\begin{cases}
	 T(r,g_1)=o\bigg(T\left(r,e^{p-(-h)}\right)\bigg),\vspace{1.2mm}\\ T(r,g_2)=o\bigg(T\left(r,e^{p-h_0}\right)\bigg),\vspace{1.2mm}\\ T(r,g_3)=o\bigg(T\left(r,e^{h-h_0}\right)\bigg).
	\end{cases}
\end{align}\par Applying Lemma \ref{lem2.8} to (\ref{e3.9}) and using (\ref{e3.8}), we obtain $g_j(z) \equiv 0$ for $j = 1, 2, 3$, which is a contradiction. Next, we suppose that $a=0$. It then follows from (\ref{e3.2}) that 
\begin{align}
\label{e3.10}{H}(z)e^{p(z)}= a_n{H}(z+nc)e^{h(z+nc)-h(z)}+\cdots+a_1{H}(z+c)e^{h(z+c)-h(z)}+a_0{H}(z).
\end{align}
Since $H \not\equiv 0$ with $\sigma(H) < k$, $\deg p = k$, and $\deg(h(z+tc) - h(z)) = k-1$ for $t=1, 2$, it follows that the order of growth of the left-hand side (L.H.S.) of (\ref{e3.10}) is exactly $k$. However, the order of growth of the right-hand side (R.H.S.) is strictly less than $k$, which leads to a contradiction.\vspace{1.5mm}

\noindent{\bfseries{Subcase 2.2.}} Let $ b_k=-a_k $. First we assume that $ a\neq 0 $. Then from (\ref{e3.2}), we obtain 
\begin{align}
\label{e3.11} & \bigg(-2a+{H}(z)e^{p(z)+h(z)}\bigg)e^{-h(z)}\\&= a_n{H}(z+nc)e^{h(z+n\eta)-h(z)}+\cdots+a_1{H}(z+c)e^{h(z+c)-h(z)}+a_0{H}(z).\nonumber 
\end{align}
\par We affirm that $ -2a+{H}(z)e^{p(z)+h(z)}\not\equiv 0 $. In fact, if $ -2a+{H}(z)e^{p(z)+h(z)}\equiv 0$, then by (\ref{e3.11}), we obtain
\begin{equation*}
	a_n{H}(z+nc)e^{h(z+nc)-h(z)}+\cdots+a_1{H}(z+c)e^{h(z+c)-h(z)}+a_0{H}(z)\equiv 0,
\end{equation*}
which can be written as
\begin{equation}
	\label{e3.12} a_n{H}(z+nc)e^{h(z+nc)}+\cdots+a_1{H}(z+c)e^{h(z+c)}+a_0{H}(z)e^{h(z)}\equiv 0. 
\end{equation}
Therefore, from (\ref{e3.12}), we obtain
\begin{equation*}
	a_nf(z+nc)+\cdots+a_1f(z+c)+a_0f(z)=a(a_n+\cdots+a_1+a_0),
\end{equation*}
which implies that $L^n_{\eta}(f)$ is a constant. This contradicts the assumption that $L^n_{\eta}(f)$ is non-constant. Consequently, we must have $\deg(p+h)\leq k-1$. Given that $\deg(-h)=k$, $\deg(h(z+tc)-h(z))=k-1$ for $t=1, \dots, n$, and $\sigma(\mathcal{H})<k$, it follows from (\ref{e3.12}) that the order of growth of the left-hand side (L.H.S.) is $k$, while the order of growth of the right-hand side (R.H.S.) is strictly less than $k$. This discrepancy yields a contradiction.\vspace{1.5mm}

\par  Assuming $a=0$, we can rewrite (\ref{e3.2}) in the form of (\ref{e3.10}). Following the same argument as in the proof of Subcase 2.1, we arrive at a contradiction.\vspace{1.5mm}

\noindent{\bfseries{Subcase 2.3.}} Suppose $b_k \neq \pm a_k$. We first consider the case $a \neq 0$, where (\ref{e3.2}) can be rewritten as (\ref{e3.8}). Since $b_k \neq a_k$ and $b_k \neq -a_k$, it follows that $\deg(p+h) = k$, $\deg(p-h) = k$, and $\deg(h+h_0) = k$. Applying the same argument as in Subcase 2.1, we deduce that $g_j(z) \equiv 0$ for $j=1, 2, 3$, which yields a contradiction.\vspace{1.5mm}

\par Assuming $a=0$, we can rewrite (\ref{e3.2}) in the form of (\ref{e3.10}). Following the same argument as in Subcase 2.1, we arrive at a contradiction. Consequently, the polynomial $p(z)$ must be a constant. We thus obtain 
\begin{equation}
	\label{e3.13} \frac{{L}^n_{c}(f)-a}{f-a}={A},
\end{equation}
where $ {A} $ is a non-zero constant. \vspace{1.5mm}

\noindent{\bf Step II:} We prove that the actual value of the constant $ a $ is zero. Suppose $ a\neq 0 $. Therefore, by (\ref{e3.1}) and (\ref{e3.13}), we deduce that 
\begin{align}
	\label{e3.14} & a_n{H}(z+nc)e^{h(z+nc)-h(z)}+\cdots+a_1{H}(z+c)e^{h(z+c)-h(z)}+(a_0-{A}){H}\\ &=\nonumber  a(1-a_0-a_1-\cdots-a_n)e^{-h(z)}. 
\end{align}
\par Since $\deg(h(z+tc)-h(z)) = k-1$ for $t=1, 2, \ldots, n$, $\sigma(H) < k$, and $\deg(h) = k$, it follows that the order of growth of the left-hand side (L.H.S.) of (\ref{e3.14}) is strictly less than $k$, whereas the order of growth of the right-hand side (R.H.S.) is exactly $k$. This yields a contradiction.\vspace{1.5mm}

\par Hence,  we deduce that $ a=0 $, and it follows from (\ref{e3.13}) that $ {L}^n_{c}(f)={A} f $ which can be written as   
\begin{align}
	\label{e3.15} a_nf(z+n\eta
	)+\cdots+a_1f(z+\eta)+a_0f(z)={A}f(z). 
\end{align}
\par Since $L^n_{\eta}(f)$ and $f$ share the value $a$ CM, it follows from Lemma \ref{lem2.5} that $f$ cannot be a rational function. Therefore, $f$ must be a transcendental meromorphic function of finite order, as must be the function $L^n_{\eta}(f)$. It remains only to determine the general solution of the difference equation (\ref{e3.15}) under the stated assumptions. The general solution of a linear difference equation of order $n$ is a solution involving $n$ periodic constants. Note that equation (\ref{e3.15}) is homogeneous with constant coefficients $a_0, a_1, \ldots, a_n$. The solutions can be obtained using Lemma \ref{lem-3.7}, with a slight modification to the argument found in \cite{Ahamed-IJPAM-2022}; hence, we omit the details.\vspace{1.5mm}
\end{proof}
\begin{proof}[\bf Proof of Theorem \ref{th1.05}]
Since $f$ is a transcendental entire function of finite order, it follows from Lemma \ref{lem2.3} that $L^n_{\eta}(f)$ is also a transcendental entire function of finite order. Furthermore, since $L^n_{\eta}(f)$ and $f$ share the value $c$ CM, we have
\begin{equation}
	\label{e3.17} \frac{ L^n_{\eta}(f)-c}{f(z)-c}=e^{Q(z)},
\end{equation}
where $Q(z)$ is a polynomial. We aim to show that $Q(z)$ is a constant. To the contrary, suppose that $\deg(Q) \geq 1$. Differentiating both sides of (\ref{e3.17}), we obtain 
\begin{equation}
	\label{e3.18} {Q}^{\prime}(z)=\frac{\left({L}^n_{\eta}(f)\right)^{\prime}}{{L}^n_{\eta}(f)-c}-\frac{f^{\prime}(z)}{f(z)-c}.
\end{equation}
It follows from (\ref{e3.18}) that $ m\left(r,{Q}^{\prime}(z)\right)=S(r,f). $ Re-writing (\ref{e3.18}), we obtain
\begin{equation}
	\label{e3.19} {Q}^{\prime}(z)=(f(z)-\zeta)\bigg(\frac{1}{f(z)-\zeta}\frac{\left({L}^n_{\eta}(f)\right)^{\prime}}{{L}^2_{\eta}(f)-c}-\frac{1}{f(z)-\zeta}\frac{f^{\prime}(z)}{f(z)-c}.\bigg). 
\end{equation}
From (\ref{e3.18}), we obtain 
\begin{equation}
	\label{e3.20} {Q}^{\prime}(z)=\frac{f(z)-\zeta}{c}{F}-\frac{f(z)-\zeta}{c-\zeta}{G},
\end{equation}
where 
\begin{equation}
	\label{e3.21} {F}=\frac{\left({L}^n_{\eta}(f)\right)^{\prime}}{f(z)-\zeta}\frac{{L}^n_{\eta}(f)}{{L}^n_{\eta}(f)-c}-\frac{\left({L}^n_{\eta}(f)\right)^{\prime}}{f(z)-\zeta}.
\end{equation}
\begin{equation}
	\label{e3.22} {G}=\frac{f^{\prime}(z)}{f(z)-\zeta}-\frac{f^{\prime}(z)}{f(z)-c}.
\end{equation} 
Since $ c\neq \zeta $, it follows from (\ref{e3.22}) that $ {G}\not\equiv 0 $. Note that $ {Q}^{\prime}(z)\not\equiv 0 $, and therefore, by (\ref{e3.20}), we obtain 
\begin{equation}
	\label{e3.23} \frac{1}{f(z)-\zeta}=\frac{(\zeta-c)\;{F}-c\;{G}}{c(\zeta-c){Q}^{\prime}(z)}.
\end{equation}
Since $ {Q}^{\prime}(z) $ is a polynomial, we obtain 
\begin{equation}
	\label{e3.24} m\left(r,\frac{1}{f(z)-\zeta}\right)\leq m(r,{F})+m\left(r,{G}\right)+S(r,f). 
\end{equation}
Since $ a_2+\cdots+a_1+a_0=0 $, employing {Lemma \ref{lem2.5}}, we obtain 
\begin{align*}
	m\left(r,{F}\right)&\leq m\left(r,\frac{{L}^n_{\eta}(f)}{f(z)-\zeta}\right)+m\left(r,\frac{\left({L}^n_{\eta}(f)\right)^{\prime}}{{L}^2_{\eta}(f)-c}\right)+m\left(r,\frac{\left({L}^n_{\eta}(f)\right)^{\prime}}{f(z)-\zeta}\right)+O(1)\\&= m\left(r,\frac{{L}^n_{\eta}(f)}{f(z)-\zeta}\right)+S(r,f)\\&=m\left(r,a_n\frac{f(z+n\eta)}{f(z)-\zeta}+\cdots+a_1\frac{f(z+\eta)}{f(z)-\zeta}+a_0\frac{f(z)}{f(z)-\zeta}\right)+S(r,f)\\&=m\left(r,a_n\frac{f(z+n\eta)-\zeta}{f(z)-\zeta}+\cdots+a_1\frac{f(z+\eta)-\zeta}{f(z)-\zeta}+a_0\frac{f(z)-\zeta}{f(z)-\zeta}\right)+S(r,f)\\&\leq  m\left(r,\frac{f(z+n\eta)-\zeta}{f(z)-\zeta}\right)+\cdots+m\left(r,\frac{f(z+\eta)-\zeta}{f(z)-\zeta}\right)+m\left(r,\frac{f(z)-\zeta}{f(z)-\zeta}\right)\\&\quad+S(r,f)\\&\leq S(r,f).
\end{align*}
\noindent It is also easy to see that

 \begin{align*}
	m\left(r,{G}\right)&= m\left(r,\frac{f^{\prime}(z)}{f(z)-\zeta}-\frac{f^{\prime}(z)}{f(z)-c}\right)\\ &\leq m\left(r,\frac{f^{\prime}(z)}{f(z)-\zeta}\right)+m\left(r,\frac{f^{\prime}(z)}{f(z)-c}\right)+O(1) \\&\leq S(r,f).
\end{align*}
Thus, we obtain
\begin{equation*}
	m\left(r,\frac{1}{f(z)-\zeta}\right)=S(r,f),
\end{equation*}
which implies that $ \delta(\zeta;f)=0 $. Therefore, $\zeta$ cannot be a Nevanlinna exceptional value of $f(z)$, which contradicts our assumption. \vspace{1.5mm}

\noindent Therefore, we conclude that $Q(z)$ is a constant. Let $e^{Q(z)} = B$, where $B \in \mathbb{C} \setminus \{0\}$. From (\ref{e3.17}), we obtain
\begin{equation*}
	\frac{L^n_{\eta}(f) - c}{f(z) - c} = B.
\end{equation*}
It remains to show that the constant $c$ must be zero. To the contrary, suppose that $c \neq 0$. Following the same argument as in Step II, we arrive at a contradiction. Thus, we must have $c = 0$, which implies $L^n_{\eta}(f) \equiv Bf$. By applying the same method for finding the solution as in the proof of Theorem \ref{th2.2} (see \cite{Ahamed-IJPAM-2022}), we obtain the desired result. This completes the proof.
\end{proof}	

\section{\bf Concluding remarks}
\noindent The difference equation $ {L}^n_{\eta}(f)\equiv {A}f $ can be written as 
\begin{equation}
	\label{e6.1} \sum_{j=0}^{2}b_jf(z+j\eta)\equiv 0,
\end{equation}
 where $ b_n=a_n $, $ \ldots $, $ b_1=a_1 $ and $ b_0=a_0-{A} $ are constants.Since (\ref{e6.1}) is a homogeneous difference equation, it follows that the solution $f$ cannot be a rational function; thus, $\sigma(f) \geq 1$. As established previously, the meromorphic solutions of the difference equation $L^n_{\eta}(f) \equiv Af$ can be obtained explicitly. \vspace{1.5mm}
 
 \par  However, investigating non-trivial solutions for the difference equation (\ref{e6.1}) becomes more complex in the non-homogeneous case, where the equation takes the form
 \begin{equation}
 	\label{e6.2} \sum_{j=0}^{n}b_j(z)f(z+j\eta)\equiv b(z),
 \end{equation}
  where the coefficients $ b_0(z), b_1(z) $, $ \ldots $, $ b_n(z) $ and $ b(z) $ are all  polynomials. 
\begin{rem}
	It is evident that when the coefficients are polynomials, equation (\ref{e6.2}) admits rational solutions for the case $n=2$.
\end{rem}
\begin{exm}
	The equation 
	\begin{equation*}
		(z+3)f(z+2)+(z+2)^2f(z+1)+(z^2-1)f(z)=2z^2+3z+4
	\end{equation*}
	 has a rational solution $ f(z)={z}/{(z+1)}. $
\end{exm}
\begin{rem}
	For $ n=2 $, transcendental meromorphic solution of the equation (\ref{e6.2}) also exists in both the case when $ b(z)=0 $ as well as $ b(z)\neq 0 $.
\end{rem}
\begin{exm}
	The equation 
	\begin{equation*}
		f(z+2)+z^2f(z+1)-(z^2+1)f(z)=0 
	\end{equation*}
	 has a solution $ f(z)=\tan (\pi z) $.
\end{exm}
\begin{exm}
	The equation 
	\begin{equation*}
		f(z+2)+z^2f(z+1)-(z^2+1)f(z)=z^2+2 
	\end{equation*}
	 has a solution $ f(z)=\tan (\pi z)+z $.
\end{exm}
In $1997$, Ozawa \cite{Oza & KMJ - 1997} demonstrated that for any given order $\sigma \in [1, \infty]$, there exists an entire function $g(z)$ with period $1$ such that $\sigma(g) = \sigma$. Notably, if $\sigma \notin \mathbb{N}$, then the order and the exponent of convergence of the zeros coincide, \textit{i.e.,} $\sigma(g) = \lambda(g) = \sigma$. This result enables us to construct specific examples for \eqref{e6.2} that demonstrate the sharpness of the growth order of its solutions.
\begin{exm}
	Let $ f(z)=e^z+1 $ is of order $ \sigma(f)=\lambda(f)=1 $, and solves the equation 
	\begin{equation*}
		z^2f(z+2)-(ez^2+1)f(z+1)-ef(z)=z^2-\left(ez^2+1\right)-e,
	\end{equation*}
	 while the gamma function $ \Gamma(z) $ is of order $ \sigma(\Gamma)=\lambda\left(1/\Gamma\right)=1 $, and solves the equation 
	 \begin{equation*}
	 	\Gamma(z+1)-z\Gamma(z)=0.
	 \end{equation*}
\end{exm} 
\begin{exm}
	Let $ f_1(z)=e^z $ and $ f_2(z)=g(z)e^z $, where $ g(z) $ is a periodic function with period $ 1 $ such that $ \sigma(g)=\lambda(g)=\sigma\in (1,2) $, are solutions of the equation 
	\begin{equation*}
		z^2f(z+2)-(ez^2+1)f(z+1)-ef(z)\equiv 0.
	\end{equation*}
\end{exm}
\begin{rem}
	Even when some coefficients are transcendental entire functions, one can still find solutions of finite order for (\ref{e6.1}).
\end{rem}
\begin{exm}
We consider the functions $ f_1 $ and $ f_2 $, where $ f_1(z)=\exp\left({(z^2-1)}/{2}\right) $ is of order $ \sigma(f_1)=2 $, and $ f_2(z)=\exp\left({(z^2-1)}/{2}\right)\sin(2\pi z) $ is of order $ \sigma(f_2)=2 $. A simple computation shows that $ f_i\; (i=1,\; 2) $ solve the equation
\begin{equation*}
	f(z+1)+e^zf(z)=0.
\end{equation*}
\end{exm}
\par The above discussions motivate us to raise the following question.
\begin{ques}
Is it possible to extend Theorem \ref{th2.2} and Theorem \ref{th1.05} by considering the operator $L^n_{c}(f)$ with coefficients that are either polynomials or transcendental entire functions? 
\end{ques} 
The method employed in the proof of Theorem \ref{th1.05} makes it clear that the estimate
\begin{equation*}
	m\left(r,\frac{L^n_{c}(f)}{f(z)-\zeta}\right) = S(r,f)
\end{equation*}
holds provided that $\sum_{i=0}^{n} a_i = 0$. This observation leads to a natural question.
\begin{ques}
Keeping the other conditions intact, is it possible to prove Theorem \ref{th1.05} in the case where $\sum_{i=0}^{n} a_i \neq 0$?
\end{ques}
\vspace{5mm}

\noindent\textbf{Author Contributions Statement:}
Both authors contributed equally to the conception and design of this study. Each author has reviewed and approved the submitted version of the manuscript.\\

\noindent\textbf{ Compliance of Ethical Standards:}\\
 
\noindent\textbf{ Conflict of interest:} The authors declare that there are no competing interests associated with the publication of this manuscript. \\

\noindent\textbf{ Data availability statement:}  Data sharing not applicable to this article as no datasets were generated or analysed during the current study.


\begin{thebibliography}{99}
	
	\bibitem{Aha & RM & 2019} {\sc M. B. Ahamed}, {An investigation on the conjecture of Chen and Yi}, \textit{Results Math.} \textbf{74}(2019): 122. 
	
	\bibitem{Ahamed-IJPAM-2022} {\sc M. B. Ahamed,} The class of meromorphic functions sharing values with their difference polynomials, \textit{Indian J. Pure Appl. Math.} \textbf{54}(2023), 1158-1169.
	
	\bibitem{Abl & Hal & Her 2000} {\sc M. J. Ablowitz}, {\sc R. G. Halburd} and {\sc B. Herbst}, {On the extension of the Painleve property to difference equations}, 
	\textit{Nonlinearlty}  \textbf{13}(2000), 889-905. 
	
	\bibitem{Ban & Aha & MS & 2019} {\sc A. Banerjee} and {\sc M. B. Ahamed}, {Uniqueness of meromorphic function with its shift operator under the purview of two or three shared sets}, \textit{Math. Slovaca} \textbf{69}(3)(2019), 557-572.
	
	\bibitem{Ban & Aha & JCMA & 2020} {\sc A. Banerjee} and {\sc  M. B. Ahamed,} {Results on meromorphic function sharing two sets with its linear c-difference operator}, \textit{J. Contemp. Math. Anal.} \textbf{55}(3), 143-155 (2020).
	
	\bibitem{Ban & Kau - 1976} {\sc S. B. Bank} and {\sc R. P. Kaufman}, {An extension of H$ \rm\ddot{o} $lders theorem concerning the gamma function}, 
	\textit{Funkcial. Ekvac.} \textbf{19}(1)(1976), 53-63.
	
	\bibitem{Ber & Lan - 2007} {\sc W. Bergweiler} and {\sc J. K. Langley}, {Zeros of difference of meromorphic functions}, 
	\textit{Math. Proc. Cambridge Philos. Soc.} \textbf{142}(2007), 133-147.
	
	\bibitem{Che - 2011} {\sc Z. X. Chen}, {Some results on difference Riccati equations}, \textit{Acta Math. Sin. Eng. Ser.} \textbf{27}(2011), 1091-1100. 
	
	\bibitem{Chen JMMA 2013} {\sc Z. X. Chen}, {On properties of meromorphic solutions for difference equations concerning Gamma function}, \textit{J. Math. Anal. Appl.} \textbf{406}(2013), 147-157.
	
	\bibitem{Che & Yi-RM-2013} {\sc Z. X. Chen} and {\sc H. X. Yi}, {On sharing values of meromorphic functions and their differences}, \textit{Results Math.} \textbf{63}(2013), 557-565.
	
	\bibitem{Chi & Fen-RM-2008} {\sc Y. M. Chiang} and {\sc S. J. Feng}, {On the Nevanlinna characteristic of $ f(z+c) $ and difference equation in the complex plane,} \textit{Ramanujan J.} \textbf{16}(2008), 105-129.
	
	\bibitem{Gun - TAMS-1983} {\sc G. Gundersen}, {Meromorphic function share four values}, \textit{Trans. Amer. Math. Soc.} \textbf{277}(1983), 545-567.
	
	\bibitem{Gun - TAMS-1987} {\sc G. Gundersen}, {Correction to meromorphic functions that share four values}, \textit{Trans. Amer. Math. Soc.} \textbf{304}(1987), 847-850. 
	
	\bibitem{Hal & Kor-AASFM-2006} {\sc R. G. Halburd} and {\sc R. J. Korhonen}, {Nevanlinna theory for the difference operator},\textit{ Ann. Acad. Sci. Fenn. Math.} \textbf{31}(2006), 463-478.
	
	\bibitem{Hal & Kor - JPA-2007} {\sc R. G. Halburd} and {\sc R. J. Korhonen}, {Meromorphic solution of difference equations, integrability and the discrete}, \textit{J. Phys. A.} \textbf{40}(2007), 1-38.
	
	\bibitem{H & K & T - TAMS -2014} {\sc R. G. Halburd}, {\sc R. Korhonen}, and {\sc K. Tohge}, {Holomorphic
		curves with shift-invariant hyperplane preimages}, \textit{Trans.
	Amer. Math. Soc.} \textbf{366}(2014), 4267–4298.
	
	\bibitem{Hay-1964} {\sc W. K. Hayman}, {Meromorphic functions}, Oxford: \textit{Clarendon Press}, 1964.
	
	\bibitem{Hei & Kor & Lai -JMMA-2009} {\sc J. Heittokangas}, {\sc R. J. Korhonen}, {\sc I. Laine} and {\sc J. Rieppo}, {Value sharing results for shifts of meromorphic functions, and sufficient condition for periodicity}, \textit{J. Math. Anal. Appl.} \textbf{355}(2009), 352-363.
	
	\bibitem{Risto-TAMS-2020} {\sc Z. J. Hua} and {\sc R. Korhonen}, Studies of differences from the point of view of Nevanlinna theory, {\it Trans. Amer. Math. Soc.} \textbf{373}(6)(2020), 4285-4318.
	
	\bibitem{Jia & Chen 2013} {\sc Y. Y. Jiang} and {\sc Z. X. Chen}, {On solutions of q-difference Riccati equations with rational coefficients}, \textit{Appl. Anal. Discrete Math.} \textbf{7}(2013), 314–326.
	
	\bibitem{Li-Hao-Yi-MJM-2021} {\sc X. M. Li, C. S. Hao} and {\sc H. X. Yi}, On the growth of meromorphic solutions of certain nonlinear difference equations, {\it Mediterr. J. Math.} 18: 56 (2021).
	
	\bibitem{Liu -RM - 2010} {\sc K. Liu}, {Zeros of difference polynomials of meromorphic functions}, \textit{Results Math.} \textbf{57}(2010), 365-376.
	
	\bibitem{Liu & Yan - AM - 2009} {\sc K. Liu} and {\sc L. Z. Yang},  {Value distribution of the difference operator}, \textit{Arch. Math.} \textbf{92}(2009), 270-278.
	
	\bibitem{Lu & Lu - CMFT-2018} {\sc F. L\emph{$ \rm \ddot{U} $}} and {\sc W. L\emph{$ \rm\ddot{U} $}}, {Meromorphic functions sharing three values with their difference operators}, \textit{Comput. Methods. Funct. Theory.} {\bf 17} (2017),  395--403.
	
	\bibitem{Mallick-AAMP-2022} {\sc S. Mallick} and {\sc M. B. Ahamed},  On uniqueness of a meromorphic function and its higher difference operators sharing two sets, \textit{Anal. Math. Phys.} 12, 78(2022). https://doi.org/10.1007/s13324-022-00668-8.
	
	\bibitem{Moh-FFA-1971} {\sc A. Z. Mohonko}, {The Nevanlinna characteristics of certain meromorphic functions,} \textit{Teor. Funktsii Funktsional. Anal. i Prilozhen}, \textbf{14}(1971), 83-87 (Russian).
	
	\bibitem{Mues-2011} {\sc E. Mues}, {Meromorphic functions sharing four values}, \textit{Complex Var. Theory Appl.} \textbf{12}(1989), 169-174.
	
	\bibitem{Navan-1929} {\sc R. Nevanlinna}, {Le Th$ {\rm \acute{e}} $or$ {\rm \acute{e}} $me de Picard-Borel et la Th$ {\rm \acute{e}} $orie des Fonctions M$ {\rm \acute{e}} $romorphes}, \textit{Gauthier-Villars}, Paris (1929).
	
	\bibitem{Oza & KMJ - 1997} {\sc M. Ozawa}, {On the existence of prime periodic entire functions}, {\it Kodai Math. Sem. Rep.} \textbf{29}(1977/78)(3), 308-321.
	
	\bibitem{Rub & Yan -LN - 1977}{\sc  L. Rubel} and {\sc C. C. Yang}, {Values shared by an entire function and its derivative}, {\it  Lecture Notes in Mathematics} \textbf{599}(1977), 101-103, Berlin, Springer-Verlag.
	
	\bibitem{Shi - 1981} {\sc S. Shimomura}, {Entire solutions of a polynomial difference equations,} {\it J. Fac. Sci. Univ. Tokyo Sect. IA Math.} \textbf{28}(2)(1981), 253-266.
	
	\bibitem{Val-BSM-59} {\sc G. Valiron}, {Sur la d$\acute{e}$riv$\acute{e}$e des fonctions alg$\acute{e}$bro$\ddot{i}$des}, {\it Bull. Soc. Math. France} \textbf{59}(1931), 17-39.
	
	\bibitem{Wang JMMA 2011} {\sc J. Wang}, {Growth and poles of meromorphic solutions of some difference equations}, \textit{J. Math. Anal. Appl.} \textbf{379}(2011), 367–377.
	
	\bibitem{Whit - 1935} {\sc J. M. Whittaker}, {Interpolatory Function Theory}, \textit{Cambridge University Press}, Cambridge, 1935.
	
	\bibitem{Func Ekvac - 1980} {\sc N. Yanagihara}, {Meromorphic solutions of some difference equations}, \textit{Funkcial Ekvac}, \textbf{23}(3)(1980), 309-326.
	
	\bibitem{Yan & Yi-2003} {\sc C. C. Yang} and {\sc H. X. Yi}, {Uniqueness theory of meromorphic functions,} \textit{Kluwer Academic Publishing Group}, Dordrecht (2003).
	
	\bibitem{Yan & Yi-2006} {\sc C. C. Yang} and {\sc H. X. Yi}, {Uniqueness theory of meromorphic functions}, \textit{Beijing: Science Press}, 2006.
	
	\bibitem{Yan - 1993} {\sc L. Yang}, {Value distribution theory}, Berlin: \textit{Springer-Verlag : Science Press}, 1993.
	
	\bibitem{Zha & Lia-2014} {\sc  J. Zhang} and {\sc L. W. Liao}, {Entire function sharing some values with their difference operators}, \textit{Sci. China. Math.} \textbf{57}(10)(2014), 2143-2152.
	
	
\end{thebibliography}
\end{document}